\documentclass[11pt,reqno,a4paper]{amsart}

\usepackage{amsthm,amsfonts,amsmath,amssymb}
\usepackage[mathscr]{eucal}

\usepackage[pdftex,bookmarks=true]{hyperref}
\usepackage{enumerate}

\newcommand\partl[2]{\frac{\partial{#1}}{\partial{#2}}}

\newcommand{\bdy}{\partial}

\newcommand{\smoo}{\mathcal{C}}

\newcommand{\bcdot}{\boldsymbol{\cdot}}

\newcommand{\C}{\mathbb{C}}
\newcommand{\R}{\mathbb{R}}

\newcommand{\D}{\mathbb{D}}

\newcommand{\h}{\mathfrak{h}}

\newcommand{\abs}[1]{\left\vert#1\right\vert}

\theoremstyle{plain}
\newtheorem{thm}{Theorem}
\newtheorem{cor}[thm]{Corollary}
\newtheorem{result}[thm]{Result}
\newtheorem{lemma}[thm]{Lemma}
\newtheorem{prop}[thm]{Proposition}

\theoremstyle{definition}

\theoremstyle{remark}
\newtheorem{rmk}[thm]{Remark}

  {\end{enumerate}}

\begin{document}

\title{A note on the smoothness of the Minkowski function}

\author[P.~Haridas]{Pranav Haridas}
\address{Department of Mathematics, Indian Institute of Technology Madras, Chennai 600036, India}
\email{pranav.haridas@gmail.com}
\author[J.~Janardhanan]{Jaikrishnan Janardhanan}
\address{Department of Mathematics, Indian Institute of Technology Madras, Chennai 600036, India}
\email{jaikrishnan@iitm.ac.in}
\thanks{Jaikrishnan Janardhanan is supported by a DST-INSPIRE fellowship from the 
Department of Science and Technology, India.}

\begin{abstract}
  The Minkowski function is a crucial tool used in the study of
  balanced domains and, more generally, quasi-balanced domains in
  several complex variables. If a quasi-balanced domain is
  bounded and pseudoconvex then it is well-known that its Minkowski
  function is plurisubharmonic. In this short note, we prove that under the
  additional assumption of smoothness of the boundary, the Minkowski
  function of a quasi-balanced domain is in fact smooth away from the
  origin.  This allows us to construct a smooth plurisubharmonic
  defining function for such domains. Our result is new even in the
  case of balanced domains.
\end{abstract}

\maketitle

\section{Introduction}

The study of holomorphic mappings between balanced and quasi-balanced
domains pose an interesting challenge. As the automorphism group
contains the circle, such domains possess symmetry that often confers
strong rigidity on holomorphic mappings between these domains.
Indeed, a classical result of Cartan exploits the circle
action to show that any automorphism of a bounded balanced domain
fixing the origin must be linear.  One of the key tools that facilitate the 
study of balanced and quasi-balanced domains is the Minkowski function.
Several generalizations of Cartan's theorem are now known 
(\cite{bell1982proper,berteloot2000cartan,kosinski2014minkowski,yamamori2017origin}),
and many of them use the Minkowski function as a central tool in the
proofs.
The demand of the presence of a circle action is also
not too severe and there are several interesting classes of domains
that are quasi-balanced. For instance, the symmetrized polydisk and
related domains are quasi-balanced domains that have been extensively studied 
using the Minkowski function (see 
\cite{nikolov2006symmetrized,kosinski2011geometry}).

Let $p_1,p_2,\dots,p_n$ be relatively prime positive integers. We say
that a domain $D\subset \C^n$ is \emph{$(p_1,p_2,\dots,p_n)$-balanced
  (quasi-balanced)} if
\begin{align*}
  \lambda \bullet z \in D \ \forall
  \lambda\in \overline{\D} \ \ \forall z \in D,
\end{align*}
where $\overline{\D} $ is the closed unit disk in $\C$ and for
$z = (z_1, z_2, \ldots, z_n) \in D$, we define
$ \lambda \bullet z := (\lambda^{p_1}z_1,\lambda^{p_2}z_2,
\dots,\lambda^{p_n}z_n)$.  If $p_1 = p_2 = \dots = p_n = 1$ above,
then we say $D$ is a \emph{balanced domain} 
(also known as a \emph{complete
circular domain} in the literature).

Given a $(p_1,p_2,\dots,p_n)$-balanced domain $D\subset \C^n$, we
define the Minkowski function $\h_D : \C^n \rightarrow \C$
\begin{align*}
  \h_D(z) := \inf\{t > 0:  \frac{1}{t} \bullet z \in D \}.
\end{align*}
Clearly $D = \{z \in \C^n: \h_D(z) < 1\}$ and
$\h_D(\lambda\bullet z) = |\lambda|\h_D(z)$.  It also turns out that
$\h_D$ is plurisubharmonic if $D$ is additionally psuedoconvex. This
fact has been a crucial ingredient in several results on balanced
domains; see \cite{MR1760787,jarnicki2013invariant}, for instance. 

One natural question that seems to be unanswered (to the best of the
authors' knowledge) in the literature is the following:

\begin{center}
  \emph{Is the Minkowski function of a smoothly bounded psuedoconvex
    quasi-balanced domain a smooth function near the boundary?}
\end{center}
In fact, we found a remark in \cite[p.~190]{graham2003geometric}, with
a reference to Hamada's paper \cite{MR1760787}, stating that the
answer to the above  question is \emph{no} if the domain has
only a $\smoo^1$-boundary. That the Minkowski function of a balanced and
bounded pseudoconvex domain with $\mathcal{C}^1$-smooth
plurisubharmonic defining function is $\mathcal{C}^1$-smooth on
$\C^n \setminus \{0\}$ has already been established in
\cite[Proposition 1]{MR1760787}.  Using the recent
work \cite{MR3557783}, we are able to prove smoothness of the
Minkowski function on $\C^n \setminus \{0\}$ for any smoothly bounded
quasi-balanced domains. The main result of this paper is the following

\begin{thm} \label{THM:MINKSMOOTH} Let $D\subset \C^n$ be a smoothly
  bounded quasi-balanced pseudoconvex domain. Then the Minkowski
  function $\h_D$ is $\mathcal{C}^{\infty}$-smooth on
  $\C^n \setminus \{0\}$. Furthermore, the function
  $r(z) := \h_D(z) -1$ is a plurisubharmonic defining function for $D$.
\end{thm}
  
\begin{rmk}
  By a smoothly bounded domain, we shall mean a bounded
  domain whose boundary is $\smoo^\infty$-smooth.
\end{rmk} 

\begin{rmk}
  The analogue of the above result for convex domains is well-known. The reader
  is referred to \cite[Section~6.3]{krantz1999geometry} for details
\end{rmk}

\section{Supporting results}

Before we give the proof of Theorem~\ref{THM:MINKSMOOTH}, we first give a
brief overview of the necessary tools.

We shall now consider the setting in \cite[p.~518, p.~523]{MR3557783}.
Let $D \subset \C^n$ be a smoothly bounded domain and let 
$G\subset \textsf{Aut}(D)\cap \mathcal{C}^{\infty} (\overline{D})$ be a
compact Lie subgroup of $\textsf{Aut}(D)$ in the compact open
topology. Consider a continuous representation
$\rho : G \rightarrow GL(\C^n)$ of $G$ and the set
\begin{align*}
  \mathcal{O}(\C^n)^G :=
  \{f\in\mathcal{O}(\C^n) : f\circ \rho(g) = f
  \text{ for all } g\in G\}
\end{align*}
called the set of $G$-invariant entire functions.

A domain $D$ is said to be $G$-invariant if $\rho(g)\cdot D = D$ for
all $g\in G$.  We will say that $G$ acts transversely on $D$ if for
each $z_0 \in \bdy D$ the image of the tangent map
$d\Psi_{z_0} : T_eG \rightarrow T_{z_0}\bdy D$ associated to the map
$\Psi_{z_0}: G \rightarrow \bdy D$ given by $g \mapsto g(z_0)$, is
not contained in $T_{z_0}^{\C} \bdy D$, the complex tangent space to
$\bdy D$ at $z_0$. We have the following

\begin{result}[Theorem 2.7 in \cite{MR3557783}]\label{THM:NZZ}
  Let $G$ be a compact Lie group, which acts linearly on $\C^n$ with
  $\mathcal{O}(\C^n)^G = \C$. If $ D$ is a $G$-invariant smoothly bounded
  pseudoconvex domain in $\C^n$ that contains the origin, then $G$ acts
  transversely on $D$.
\end{result}

\noindent Consider the representation of the compact lie group
$\mathbb{S}^1$ given by
\begin{align*}
  \rho(\lambda)(z) =  \lambda\bullet z \text{ where }
  \lambda \in \mathbb{S}^1.
\end{align*}

\begin{prop}
  Under the above action, $\mathcal{O}(\C^n)^{\mathbb{S}^1} = \C$.
\end{prop}
\begin{proof}
  Consider $f\in \mathcal{O}(\C^n)$ such that
  $f(\lambda\bullet z) = f(z)$ for every $\lambda \in \mathbb{S}^1$
  and for all $z\in D$. Fix $z\in \C^n$ and define a function
  $g_{z} : \C \rightarrow \C$ given by
  $g_{z}(\lambda) = f(\lambda\bullet z)$. Then $g$ is a holomorphic
  function that is constant on $\mathbb{S}^1$ and hence
  $g_z \equiv g(0)$. Since our choice of $z$ was arbitrary, we have
  $f(z) = f(0)$.  The constant functions clearly belong to
  $\mathcal{O}(\C^n)^{\mathbb{S}^1}$.
\end{proof}

That $D$ is $\mathbb{S}^1$-invariant is a direct consequence of the
fact that $D$ is $(p_1,p_2,\dots,p_n)$-balanced.  Thus in our case, we
can conclude the following
\begin{cor} \label{COR:NZZ} Under the hypotheses on $D$ as in 
Theorem~\ref{THM:MINKSMOOTH}, for each
  $\xi = (\xi_1,\dots, \xi_n) \in \bdy D$, the vector
  \begin{align*}
    (ip_1\xi_1,\dots,ip_n\xi_n) \not \in T^\C_{\xi} \bdy D.
  \end{align*}
\end{cor}
\begin{proof}
  With $\Psi_{\xi}: \mathbb{S}^1 \rightarrow \bdy D$ given by
  $\Psi_{\xi}(\lambda) = \lambda\bullet \xi$, the evaluation of the
  derivative map $d\Psi_{\xi}(1) = (ip_1\xi_1,\dots,ip_n\xi_n) 
  \in T_{\xi} \bdy D$. By Result~\ref{THM:NZZ},
  $d\Psi_{\xi}(1)\notin T^\C_{\xi} \bdy D$ as otherwise
  $d\Psi_{\xi}(T_e\mathbb{S}^1)\subset T^\C_{\xi} \bdy D$.
\end{proof}

We will use the following version of Hopf's lemma in the proof of Theorem~\ref{THM:MINKSMOOTH}.

\begin{lemma}[Lemma 3, p.~177, \cite{khenkin89several}]\label{LEM:HOPF}
  Let $D \subset \C^n$ be a smoothly bounded domain and let $r$ be a negative
  plurisubharmonic function defined on $D$. Then there exists a constant 
  $c > 0$ such that $\abs{r(z)} > c \cdot dist(z, \bdy D)$.
\end{lemma}

\section{Proof of Theorem \ref{THM:MINKSMOOTH}}
\noindent Let $\psi$ be a defining function for $D$.  Consider the map
$g \in \mathcal{C}^{\infty}(\C^n\times \R\setminus \{0\})$ given by
\begin{align*}
  g(z,t) : = \psi\left(\frac{1}{t}\bullet z\right)
\end{align*}
Observe that $g(z, \h_D(z)) = 0$. Let us fix a point
$z_0\in \C^n\setminus \{0\}$.  We shall show that
$\partl{g}{t}|_{(z_0,\h_D(z_0))} \neq 0$.

\medskip

Let us denote the coordinates of $z_0$ by $(z_1, \ldots, z_n)$. Then
the point $\xi = (\xi_1, \ldots, \xi_n)$ defined to be
$\frac{1}{\h_D(z_0)}\bullet z_0$
belongs to $\bdy D$.  A direct calculation gives us that
\begin{align*}
  \partl{g}{t}|_{(z_0, \h_D(z_0))} =
  \frac{-1}{\h_D(z_0)}
  \begin{pmatrix}
    \partl{\psi}{z_1}, \ldots, \partl{\psi}{z_n}
  \end{pmatrix}|_{\xi} \bcdot
  \begin{pmatrix}
    p_1\xi_1 \\
    \vdots   \\
    p_n\xi_n
  \end{pmatrix}
\end{align*}
If $\partl{g}{t}|_{(z_0,\h_D(z_0))} = 0$, then
$\left(p_1\xi_1, \ldots, p_n\xi_n\right) \in T_{\xi}\bdy D$. 
Consider the curve \[\gamma(\theta) = e^{i\theta}\bullet\xi\]
in $\bdy D$. Then the corresponding tangent vector
$\left(ip_1\xi_1, \ldots, ip_n\xi_n\right) \in T_{\xi}\bdy D$ and
hence is in the complex tangent space $T_{\xi}^{\C}\bdy D$ which is a
contradiction to Corollary~\ref{COR:NZZ}.  Now by the implicit
function theorem, $\h_D$ is $\mathcal{C}^{\infty}$-smooth on
$\C^n \setminus \{0\}$.

%


\medskip

We shall now prove that $r$ is a defining function. We are left with
observing that $dr \neq 0$ on $\bdy D$. It is easy to see that the
normal derivative at every point on the boundary $\bdy D$ is bounded
below by the constant $c$ by an application of Hopf's lemma
(Lemma~\ref{LEM:HOPF}). Hence $dr \neq 0$ on $\bdy D$. 
\qed
\medskip

Our result implies that the main results 
in \cite{MR1760787,hamada2001convex} on balanced domains with $\smoo^1$-smooth
plurisubharmonic defining function also hold for smoothly bounded balanced
pseudoconvex domains.

\bibliographystyle{amsalpha} \bibliography{minkowski}
\end{document}